\pdfoutput=1
\documentclass[11pt,reqno]{amsart}
\usepackage[paperheight=279mm,paperwidth=18cm,textheight=26cm,textwidth=14cm,includehead]{geometry}
\usepackage[T1]{fontenc}
\usepackage[utf8]{inputenc}
\usepackage[unicode]{hyperref}
\hypersetup{
bookmarks=true,
colorlinks=true,
citecolor=[rgb]{0,0,0.5},
linkcolor=[rgb]{0,0,0.5},
urlcolor=[rgb]{0,0,0.75},
pdfpagemode=UseNone,
pdfstartview=FitH,
pdfdisplaydoctitle=true,
pdftitle={Weighted Lépingle inequality},
pdfauthor={Zorin-Kranich},
pdflang=en-US
}

\usepackage[style=alphabetic,maxalphanames=4]{biblatex}
\usepackage{amssymb}
\usepackage{amsmath}
\usepackage{amsthm}
\usepackage{amsfonts}
\usepackage{mathtools}

\DeclarePairedDelimiter\abs{\lvert}{\rvert}
\DeclarePairedDelimiter\norm{\lVert}{\rVert}

\DeclarePairedDelimiterX\innerp[2]{\langle}{\rangle}{#1,#2}

\providecommand\given{}
\newcommand\SetSymbol[1][]{%
\nonscript\:#1\vert
\allowbreak
\nonscript\:
\mathopen{}}
\DeclarePairedDelimiterX\Set[1]\{\}{%
\renewcommand\given{\SetSymbol[\delimsize]}
#1
}
\DeclarePairedDelimiterXPP\EE[1]{\mathbb{E}}{\lparen}{\rparen}{}{\renewcommand\given{\SetSymbol[\delimsize]}#1} 

\numberwithin{equation}{section}
\theoremstyle{plain}
\newtheorem{theorem}{Theorem}
\numberwithin{theorem}{section}

\newtheorem{lemma}[theorem]{Lemma}
\newtheorem{corollary}[theorem]{Corollary}

\theoremstyle{remark}
\newtheorem{remark}[theorem]{Remark}

\theoremstyle{definition}
\newtheorem{definition}[theorem]{Definition}

\newcommand{\dif}{\mathop{}\!\mathrm{d}} 
\newcommand{\calF}{\mathcal{F}}
\newcommand{\calX}{\mathcal{X}}
\newcommand{\R}{\mathbb{R}}
\newcommand{\N}{\mathbb{N}}

\title{Weighted L\'epingle inequality}
\author{Pavel Zorin-Kranich}
\thanks{PZ was partially supported by the Hausdorff Center for Mathematics (DFG EXC 2047)}
\address{Mathematical Institute, University of Bonn, Bonn, Germany}

\begin{document}
\begin{abstract}
We prove an estimate for weighted $p$-th moments of the pathwise $r$-variation of a martingale in terms of the $A_{p}$ characteristic of the weight.
The novelty of the proof is that we avoid real interpolation techniques.
\end{abstract}
\maketitle

\section{Introduction}
L\'epingle's inequality \cite{MR0420837} is a moment estimate for the pathwise $r$-variation of martingales.
Finite $r$-variation is a parametrization-invariant version of H\"older continuity of order $1/r$ and plays a central role in Lyons's theory of rough paths \cite{MR1654527}.

L\'epingle's inequality also found applications in ergodic theory \cite{MR1019960} and harmonic analysis \cite{MR2653686}, see \cite{arxiv:1808.04592} and \cite{MR3623404,MR3829751} and references therein, respectively, for recent developments in these directions.
Weighted inequalities in harmonic analysis go back to \cite{MR0293384}, and weighted variational inequalities have been studied since \cite{MR2501353}.
A major motivation of the weighted theory is the Rubio de Francia extrapolation theorem that allows to obtain vector-valued $L^{p}$ inequalities for all $1<p<\infty$ from scalar-valued weighted $L^{p}$ inequalities for a single $p$, see \cite[Section 3]{MR2754896} for the most basic version of that result and \cite[Theorem 8.1]{MR3667592} for a version applicable to martingales.

In this article, we prove a weighted version of L\'epingle's inequality for martingales with asymptotically sharp dependence on the $A_{p}$ characteristic of the weight.
For dyadic martingales, weighted variational inequalities were first obtained in \cite[Lemma 6.1]{MR3000985} using the real interpolation approach as in \cite{MR933985,MR1019960,MR2434308,arxiv:1808.04592}.
The argument in the dyadic case relied on the so-called open property of $A_{p}$ classes, see e.g.\ \cite[Theorem 1.2]{MR2990061}, that is in general false for martingale $A_{p}$ classes, see the example in \cite[\textsection 3]{MR544802} and \cite{MR497663}.
Therefore, we use a new stopping time argument that is also simpler than the previous proofs of L\'epingle's inequality even in the classical, unweighted, case.

\subsection{Notation}
Let $(\Omega,(\calF_{n})_{n=0}^{\infty},\mu)$ be a filtered probability space and $\calF_{\infty} := \vee_{n=0}^{\infty} \calF_{n}$.
A \emph{weight} is a positive $\calF_{\infty}$-measurable function $w : \Omega \to (0,\infty)$.
The corresponding weighted $L^{p}$ norm is given by $\norm{ X }_{L^{p}(\Omega,w)} := \bigl( \int_{\Omega} \abs{X}^{p} w \dif\mu \bigr)^{1/p}$.
For $1<p<\infty$, the \emph{martingale $A_{p}$ characteristic} of the weight $w$ is defined by
\[
Q_{p}(w) := \sup_{\tau} \norm{ \EE{w \given \calF_{\tau}} \EE{w^{-1/(p-1)} \given \calF_{\tau}}^{p-1} }_{L^{\infty}(w)},
\]
where the supremum is taken over all adapted stopping times $\tau$.
For comparison of our main result with the unweighted case, note that for $w\equiv 1$ we have $Q_{p}(w)=1$ for all $1<p<\infty$.

For $0<r<\infty$, a sequence of random variables $X = (X_{n})_{n}$, and $\omega\in\Omega$, the $r$-variation of $X$ at $\omega$ is defined by
\begin{equation}
\label{eq:Vr}
V^{r}X(\omega) := V^{r}_{n}X_{n}(\omega) :=
\sup_{u_{1} < u_{2} < \dotsb} \Bigl( \sum_{j} \abs{X_{u_{j-1}}(\omega)-X_{u_{j}}(\omega)}^{r} \Bigr)^{1/r},
\end{equation}
where the supremum is taken over arbitrary increasing sequences.

\subsection{Main result}
For an integrable $\calF_{\infty}$-measurable function $X : \Omega \to \R$, the associated martingale is defined by $X_{n} := \EE{X \given \calF_{n}}$.
We have the following weighted moment estimate for the pathwise $r$-variation of this martingale.
\begin{theorem}
\label{thm:weighted-lepingle}
For every $1<p<\infty$, there exists a constant $C_{p}<\infty$ such that, for every $r>2$, every filtered probability space $\Omega$, every weight $w$ on $\Omega$, and every integrable function $X : \Omega \to \R$, we have
\begin{equation}
\label{eq:weighted-lepingle}
\norm{ V^{r}X }_{L^{p}(\Omega,w)}
\leq C_{p} \sqrt{\frac{r}{r-2}}
Q_{p}(w)^{\max(1,1/(p-1))} \norm{ X }_{L^{p}(\Omega,w)}.
\end{equation}
\end{theorem}

\begin{remark}
By the monotone convergence theorem, Theorem~\ref{thm:weighted-lepingle} extends to c\`adl\`ag martingales.
\end{remark}

\begin{remark}
The example in \cite[Theorem 2.1]{MR1640349} shows that, for $p=2$, the constant in \eqref{eq:weighted-lepingle} must diverge at least as
\begin{equation}
\label{eq:conj-growth-rate}
\sqrt{ \log \frac{r}{r-2} } \quad\text{when}\quad r\to 2.
\end{equation}
Indeed, it is proved there that, if $(X_{n})_{n=0}^{N}$ is a martingale with i.i.d.\ increments that are Gaussian random variables with zero expectation and unit variance, then $(V^{2}X)^{2} \geq c N \log \log N$ with probability converging to $1$ as $N\to\infty$ for every $c<1/12$.
In this case, choosing $r$ such that $r-2 = 1/\log N$, by H\"older's inequality, we obtain
\[
V^{2}X
\leq
N^{1/2-1/r} V^{r}X
\leq
C V^{r}X.
\]
This would lead to a contradiction if the constant in \eqref{eq:weighted-lepingle} diverges slower than stated in \eqref{eq:conj-growth-rate}.
The growth rate of the constant in \eqref{eq:weighted-lepingle} as $r\to 2$ is important e.g.\ in Bourgain's multi-frequency lemma, as explained in \cite[\textsection 3.2]{MR3280058}.
\end{remark}

\begin{remark}
The growth rate of the constant in \eqref{eq:weighted-lepingle} as $r\to 2$ is also related to endpoint estimates, in which the $\ell^{r}$ norm in \eqref{eq:Vr} is replaced by an Orlicz space norm.
The results of \cite{MR295434} for the Brownian motion suggest that it might be possible to use a Young function that decays as  $x^{2}/\log\log x^{-1}$ when $x\to 0$.
Such an estimate would imply an estimate of the form \eqref{eq:conj-growth-rate} for the constant in \eqref{eq:weighted-lepingle}, and it would have useful consequences for rough differential equations, see \cite[Remark 5]{MR2387018}.
Our method allows to use Young functions that decay as $x^{2}/(\log x^{-1})^{1+\epsilon}$ when $x\to 0$.
\end{remark}

\begin{remark}
A Fefferman--Stein type weighted estimate that substitutes \eqref{eq:weighted-lepingle} in the case $p=1$ can be deduced from Corollary~\ref{cor:ptw-domination} and \cite[Theorem 1.1]{MR3567926}.
\end{remark}

\section{Stopping times and a pathwise $r$-variation bound}
In this section, we estimate the $r$-variation of an arbitrary adapted process pathwise by a linear combination of square functions.
We consider an adapted process $(X_{n})_{n}$ with values in an arbitrary metric space $(\calX,d)$ and extend the definition of $r$-variation \eqref{eq:Vr} by replacing the absolute value of the difference by the distance.
We have the following metric spaces $\calX$ in mind.
\begin{enumerate}
\item In Theorem~\ref{thm:weighted-lepingle}, we will use $\calX=\R$ (and $\rho=2$ below).
\item In applications to the theory of rough paths, one takes $\calX$ to be a free nilpotent group, see \cite[\textsection 9]{MR2604669}.
\item When $\calX$ is a Banach space with martingale cotype $\rho \in [2,\infty)$, Corollary~\ref{cor:ptw-domination} can be used to recover \cite[Theorem 4.2]{MR933985}.
\end{enumerate}
\begin{definition}
Let $M_{t} := \sup_{t'' \leq t' \leq t} d(X_{t'},X_{t''})$.
For each $m\in \N$, define an increasing sequence of stopping times by
\begin{equation}
\label{eq:stop-times}
\tau_{0}^{(m)}(\omega) := 0,
\quad
\tau_{j+1}^{(m)}(\omega) := \inf \Set{ t \geq \tau_{j}^{(m)}(\omega) \given d( X_{t}(\omega) , X_{\tau_{j}(\omega)}(\omega) ) \geq 2^{-m} M_{t}(\omega) }.
\end{equation}
\end{definition}

\begin{lemma}
\label{lem:comparable-jump}
Let $0 \leq t' < t < \infty$ and $m\geq 2$.
Suppose that
\begin{equation}
\label{eq:jump-magnitude}
2 < d( X_{t'}(\omega) , X_{t}(\omega))/(2^{-m}M_{t}(\omega)) \leq 4.
\end{equation}
Then there exists $j$ with $t' < \tau^{(m)}_{j}(\omega) \leq t$ and
\begin{equation}
\label{eq:comparable-jump}
d( X_{t'}(\omega) , X_{t}(\omega))
\leq
8 d( X_{\tau^{(m)}_{j-1}(\omega)}(\omega) , X_{\tau^{(m)}_{j}(\omega)}(\omega)).
\end{equation}
\end{lemma}
\begin{proof}
We fix $\omega$ and omit it from the notation.
Let $j$ be the largest integer with $\tau' := \tau^{(m)}_{j} \leq t$.
We claim that $\tau' > t'$.
Suppose for a contradiction that $\tau' < t'$ (the case $\tau'=t'$ is similar but easier).
By the hypothesis \eqref{eq:jump-magnitude} and the assumption that $t,t'$ are not stopping times, we obtain
\[
2 \cdot 2^{-m} M_{t}
<
d( X_{t'}, X_{t} )
\leq
d( X_{\tau'}, X_{t'}) + d( X_{\tau'}, X_{t})
<
2^{-m} M_{t'} + 2^{-m} M_{t}
\leq
2 \cdot 2^{-m} M_{t},
\]
a contradiction.
This shows $\tau' > t'$.

It remains to verify \eqref{eq:comparable-jump}.
Assume that $M_{\tau'} < M_{t}/2$.
Then, for some $\tau' < \tau'' \leq t$, we have $d( X_{\tau'}, X_{\tau''}) \geq M_{t}/2 \geq 2^{-m} M_{\tau''}$, contradicting maximality of $\tau'$.
It follows that
\[
d( X_{\tau^{(m)}_{j-1}} , X_{\tau^{(m)}_{j}} )
\geq
2^{-m} M_{\tau'}
\geq
2^{-m} M_{t}/2
\geq
d( X_{t'} , X_{t})/8.
\qedhere
\]
\end{proof}

\begin{lemma}
\label{lem:ptw-domination}
For every $0<\rho<r<\infty$, we have the pathwise inequality
\begin{equation}
\label{eq:ptw-domination}
V^{r}_{t}(X_{t}(\omega))^{r}
\leq
8^{\rho} \sum_{m=2}^{\infty} \bigl( 2^{-(m-2)} M_{\infty}(\omega) \bigr)^{r-\rho} \sum_{j=1}^{\infty}
d( X_{\tau^{(m)}_{j-1}(\omega)}(\omega) , X_{\tau^{(m)}_{j}(\omega)}(\omega) )^{\rho}.
\end{equation}
\end{lemma}
\begin{proof}
We fix $\omega$ and omit it from the notation.
Let $(u_{l})$ be any increasing sequence.
For each $l$ with $d( X_{u_{l}} , X_{u_{l+1}})\neq 0$, let $m=m(l) \geq 2$ be such that
\[
2 < d(  X_{u_{l}} , X_{u_{l+1}} )/(2^{-m}M_{u_{l+1}}) \leq 4.
\]
Such $m$ exists because the distance is bounded by $M_{u_{l+1}}$.

Let $j$ be given by Lemma~\ref{lem:comparable-jump} with $t'=u_{l}$ and $t=u_{l+1}$.
Then
\[
d(  X_{u_{l}} , X_{u_{l+1}} )^{r}
\leq
8^{\rho} d( X_{\tau^{(m)}_{j-1}} , X_{\tau^{(m)}_{j}})^{\rho}
\cdot (4 \cdot 2^{-m} M_{u_{l+1}})^{r-\rho}.
\]
Since each pair $(m,j)$ occurs for at most one $l$, this implies
\[
\sum_{l} d(  X_{u_{l}} , X_{u_{l+1}} )^{r}
\leq
8^{\rho} \sum_{m,j} d( X_{\tau^{(m)}_{j-1}} , X_{\tau^{(m)}_{j}})^{\rho}
\cdot (2^{-(m-2)} M_{\infty})^{r-\rho}.
\]
Taking the supremum over all increasing sequences $(u_{l})$, we obtain \eqref{eq:ptw-domination}.
\end{proof}

\begin{corollary}
\label{cor:ptw-domination}
For every $0<\rho<r<\infty$, we have the pathwise inequality
\begin{equation}
\label{eq:ptw-domination:only-square}
\boxed{
V^{r}_{t}(X_{t}(\omega))^{\rho}
\leq
8^{\rho} \sum_{m=2}^{\infty} 2^{-(m-2)(r-\rho)} \sum_{j=1}^{\infty}
d( X_{\tau^{(m)}_{j-1}(\omega)}(\omega) , X_{\tau^{(m)}_{j}(\omega)}(\omega) )^{\rho}.
}
\end{equation}
\end{corollary}
\begin{proof}
By the monotone convergence theorem, we may assume that $X_{n}$ becomes independent of $n$ for sufficiently large $n$.
In this case,
\[
M_{\infty}(\omega) \leq V^{r}_{t}(X_{t}(\omega)) < \infty.
\]
Substituting this inequality in \eqref{eq:ptw-domination} and canceling $V^{r}_{t}(X_{t}(\omega))^{r-2}$ on both sides, the claim follows.
\end{proof}

\section{Proof of the weighted L\'epingle inequality}
Estimates in weighted spaces $L^{p}(\Omega,w)$ for differentially subordinate martingales with sharp dependence on the characteristic $Q_{p}(w)$ were obtained in \cite{MR3406523} in the discrete case (a simpler alternative proof is in \cite{MR3625108}) and \cite{MR3916937} in the continuous case (a simpler alternative proof is in \cite{arxiv:1607.06319}).
By Khintchine's inequality, these results imply the following weighted estimate for the martingale square function.

\begin{theorem}[{cf.~\cite{arxiv:1607.06319}}]
\label{thm:square}
Let $(X_{j})_{j=0}^{\infty}$ be a martingale on a probability space $\Omega$.
Then, for every $1<p<\infty$, we have
\begin{equation}
\label{eq:square}
\norm[\Big]{ \bigl( \sum_{j=1}^{\infty} \abs{X_{j}-X_{j-1}}^{2} \bigr)^{1/2} }_{L^{p}(\Omega,w)}
\leq C_{p}
Q_{p}(w)^{\max(1,1/(p-1))} \norm{ X }_{L^{p}(\Omega,w)},
\end{equation}
where the constant $C_{p}<\infty$ depends only on $p$, but not on the martingale $X$ or the weight $w$.
\end{theorem}
An alternative proof that deals directly with the square function \eqref{eq:square} appears in \cite{MR3748572}, but it is carried out only for continuous time martingales with continuous paths.

\begin{proof}[Proof of Theorem~\ref{thm:weighted-lepingle}]
By extrapolation, see \cite[Theorem 8.1]{MR3667592}, it suffices to consider $p=2$.
We will in fact give a direct proof for $2 \leq p < \infty$.
A similar argument also works for $1 < p < 2$, but gives a poorer dependence on $r$ than claimed in \eqref{eq:weighted-lepingle}.

Let $\tau^{(m)}_{j}$ be the stopping times constructed in \eqref{eq:stop-times}, and let
\[
S_{(m)}(\omega) := \Bigl( \sum_{j=1}^{\infty} \abs{X_{\tau^{(m)}_{j-1}(\omega)}(\omega) - X_{\tau^{(m)}_{j}(\omega)}(\omega)}^{2} \Bigr)^{1/2}
\]
denote the square function of the sampled martingale $(X_{\tau^{(m)}_{j}})_{j}$.
Then Corollary~\ref{cor:ptw-domination} with $\calX=\R$ and $\rho=2$ gives
\[
V^{r}X
\leq
8 \Bigl( \sum_{m=2}^{\infty} 2^{-(m-2)(r-2)} S_{(m)}^{2} \Bigr)^{1/2}.
\]
Since $2 \leq p < \infty$, by Minkowski's inequality, this implies
\[
\norm{V^{r}X}_{L^{p}(\Omega,w)}
\leq
8 \Bigl( \sum_{m=2}^{\infty} 2^{-(m-2)(r-2)} \norm{S_{(m)}}_{L^{p}(\Omega,w)}^{2} \Bigr)^{1/2}.
\]
Inserting the square function estimates \eqref{eq:square} for the sampled martingales $(X_{\tau^{(m)}_{j}})_{j}$ on the right-hand side above, we obtain
\begin{align*}
\norm{ V^{r}X }_{L^{p}(\Omega,w)}
&\leq
8C_{p} Q_{p}(w) \norm{ X }_{L^{p}(\Omega,w)} \Bigl( \sum_{m=2}^{\infty} 2^{-(m-2)(r-2)} \Bigr)^{1/2}
\\ &=
8 C_{p} \bigl( 1- 2^{-(r-2)} \bigr)^{-1/2} Q_{p}(w) \norm{ X }_{L^{p}(\Omega,w)}.
\end{align*}
This implies \eqref{eq:weighted-lepingle}.
\end{proof}

\begin{remark}
One can also directly apply Theorem~\ref{thm:square} for $1<p<2$, without passing through the extrapolation theorem.
But this seems to lead to a faster growth rate of the constant in \eqref{eq:weighted-lepingle} as $r\to 2$.
\end{remark}

\begin{remark}
The unweighted L\'epingle inequality (Theorem~\ref{thm:weighted-lepingle} with $w \equiv 1$) follows from Corollary~\ref{cor:ptw-domination} and the usual Burkholder--Davis--Gundy (BDG) inequality.
\end{remark}

\begin{remark}
Corollary~\ref{cor:ptw-domination} can be used to recover the $p$-variation rough path BDG inequality \cite[Theorem 4.7]{MR3909973}.
For convex moderate functions $F(x)=x^{p}$ with $1<p<\infty$, the required estimate for the square function appearing in \eqref{eq:ptw-domination:only-square} can be deduced from the usual BDG inequality and \cite[Proposition 3.1]{arxiv:1812.09763}.
The latter result can be extended to arbitrary convex moderate functions $F$ using the Davis martingale decomposition.
\end{remark}

\begin{remark}
Let $\rho \in [2,\infty)$, and let $\calX$ be a Banach space with martingale cotype $\rho$.
Using Corollary~\ref{cor:ptw-domination} and the $\rho$-function bounds for $\calX$-valued martingales in \cite[Theorem 10.59]{MR3617459}, we see that, for every $1<p<\infty$, $r>\rho$, every filtered probability space $\Omega$, and every integrable function $X : \Omega \to \calX$, we have
\begin{equation}
\label{eq:vv-lepingle}
\norm{ V^{r}X }_{L^{p}(\Omega)}
\leq C_{\calX,p} \frac{r}{r-\rho}
\norm{ X }_{L^{p}(\Omega)}.
\end{equation}
In fact, it is possible to obtain a slightly better dependence on $r$, which we omit for simplicity.
There is also an endpoint version of \eqref{eq:vv-lepingle} at $p=1$, in which $X$ is replaced by the martingale maximal function on the right-hand side.

The vector-valued estimate \eqref{eq:vv-lepingle} was first proved in \cite[Theorem 4.2]{MR933985}, with an unspecified dependence on $r$.
The dependence on $r$ stated in \eqref{eq:vv-lepingle} can also be obtained using Theorem 1.3 and Lemma 2.17 in \cite{arxiv:1808.04592}, as well as real interpolation, but this method does not work at the endpoint $p=1$.
\end{remark}

\paragraph*{Acknowledgment.}
This work was partially supported by the Hausdorff Center for Mathematics (DFG EXC 2047).

\printbibliography
\end{document}